\documentclass[11pt]{article}
\usepackage[utf8]{inputenx}
\usepackage{fourier}
\usepackage{graphicx}
\usepackage{amsmath}
\usepackage{amsfonts}
\usepackage[english]{babel}
\usepackage{amsthm}
\usepackage{latexsym}
\usepackage{xcolor}
\usepackage{comment}
\usepackage{tikz}
\tikzset{partial ellipse/.style args={#1:#2:#3}{insert path={+ (#1:#3) arc (#1:#2:#3)}}}
\newcommand{\redline}{\raisebox{2pt}{\tikz{\draw[red,solid,thick](0.1,0) -- (0.4,0);}}}
\newcommand{\blueline}{\raisebox{2pt}{\tikz{\draw[blue,solid,thick](0.1,0) -- (0.4,0);}}}

\numberwithin{equation}{section}
\newtheorem{theorem}{Theorem}[section]
\newtheorem{assumption}{Assumption}
\newtheorem{corollary}[theorem]{Corollary}

\newtheorem{lemma}[theorem]{Lemma}
\newtheorem{proposition}[theorem]{Proposition}
\newtheorem{remark}[theorem]{Remark}
\newtheorem{problem}[theorem]{Problem}
\def\neweq#1{\begin{equation}\label{#1}}
\def\endeq{\end{equation}}

\input texdraw

\newcommand{\R}{\mathbb{R}}

\newcommand{\RE}{\mathbb{R}}

\renewcommand{\div}{{\rm div}}
\newcommand{\om}{\Omega}
\newcommand{\omdel}{\Omega_{\delta}}
\newcommand{\omdd}{\Omega\setminus{D}}
\newcommand{\omd}{\Omega_D}

\newcommand{\eps}{\epsilon}

\begin{document}

\title{On a nonlinear model in domains with cavities arising from cardiac electrophysiology }

\author{Elena Beretta \footnote{ NYU Abu Dhabi,(corresponding editor), \texttt{eb147@nyu.edu}}\ ,
		M. Cristina Cerutti \footnote{Dipartimento di Matematica, Politecnico di Milano , \texttt{crisitina.cerutti@polimi.it}}\ ,
	Dario Pierotti \footnote{Dipartimento di Matematica, Politecnico di Milano , \texttt{dario.pierotti@polimi.it}}}

\date{\today}

\maketitle

\begin{abstract} 
In this paper we deal with the problem of determining perfectly insulating regions (cavities) from boundary measurements in a nonlinear elliptic equation arising from cardiac electrophysiology. With minimal regularity assumptions on the cavities, we first show well-posedness of the direct problem and then prove uniqueness for the inverse problem.
\end{abstract}
\vskip.15cm

\vskip.15cm





\section{Introduction}
In this paper we analyze a mathematical model arising from applications in cardiac electrophysiology.
The goal is to collect boundary data and a mathematical description of the electrical activity of the heart in order to detect ischemic regions,  characterized by having conductivity properties different from the ones in the surrounding healthy tissue \cite{Pavarino_2014_book}.

In \cite{BCMP},  the stationary version of the {\it monodomain model} describing the electrical activity of the heart was investigated leading to the study of a Neumann problem for a semilinear elliptic equation. Ischemic regions were modeled as conductivity inhomogeneities of small size and with low conductivity compared to the surrounding medium. The presence of the inhomogeneity alters the electrical activity generating a perturbation in the transmembrane potential described in terms of an asymptotic expansion in the smallness parameter which contains information on the size and shape of the altered region. Based on the aforementioned results a topological gradient method was implemented in \cite{BMR} for the effective reconstruction of the inhomogeneities from boundary measurements of the potential.

In \cite{BRV} the authors analyzed the mathematical model in the case of inhomogeneities of arbitrary shape and size in the two-dimensional setting. In particular, the issue of reconstruction of the inhomogeneity from boundary measurements was addressed. It is well known that this is a highly nonlinear and severely ill-posed inverse problem. Indeed, even for its linear counterpart, the inverse {\it  conductivity problem},  uniqueness might be guaranteed only if infinitely many measurements of solutions are available and even with smoothness a-priori information on the unknown inclusion, the continuous dependence from the data is logarithmic.



It has been observed that, after myocardial infarction, and completed healing process, lethal ventricular ischemic tachycardia can appear. This is due to the presence of the infarct scar, a non-excitable tissue (mainly composed of collagen), that can be modeled as an electrical insulator \cite{Perez}, \cite{Relan},\cite{Fronteraetal}. The determination of these regions and their shape is then fundamental to perform successful radiofrequency ablation for the prevention of tachycardia. 

Mathematically, this leads to model the damaged regions as perfect
insulators, i.e. as a {\it cavities}. In this case, the inverse problem is expected to be more treatable. In fact, in the linear {\it conductivity equation} one boundary measurement is enough to recover in a stable way  smooth cavities, cf. \cite{ABRV}.  

We start with analyzing the well-posedness of the direct boundary value problem under minimal regularity assumptions on the cavities. 
Notably,  we establish a uniqueness result for the inverse problem showing that a single boundary measurement of the potential is enough to detect a finite collection of separated simply connected Lipschitz cavities.
More precisely, we consider the following boundary value problem
\begin{equation}\label{probcav1}
\left\{  \begin{array}{ll}    -\div(A(x)\nabla u)+u^3=f, & \hbox{in $\omdd\subset\mathbb{R}^2$} \\    \displaystyle{\frac{\partial^A u}{\partial\mathbf{n}}}=0, & \hbox{on $\partial(\omdd)$}\;,  \end{array}\right.\end{equation}
where, in the application we have in mind,  $A(x)$ is the {\it anisotropic conductivity tensor} in the heart tissue, $ \displaystyle{\frac{\partial^A u}{\partial\mathbf{n}}}$ is the \textit{current flux}, $D\subset\Omega$ denotes the \textit{ischemic region}. Finally, $f$ is a given instantaneous current applied to the heart tissue, usually in a confined region and expressing the {\it initial electrical stimulus} and $u=u(D)$ the {\it transmembrane potential}.

We first show well-posedness of (\ref{probcav1}) in a variational setting in $H^1(\Omega\backslash D)$, where $\Omega\backslash D$ has Lipschitz boundary, together with some crucial a priori $L^{\infty}$ bounds for the solutions that we derive using the truncation method of Stampacchia.  Subsequently,  by exploiting similar bounds for a suitable sequence of approximating problems, we prove the well-posedness of the direct problem in $H^1(\Omega\backslash D)$, with $\Omega\backslash D$ belonging to a special class of finite perimeter sets.
We point out that in non-Lipschitz domains Sobolev embeddings and extension properties do not  hold and therefore cannot be used to prove existence and uniqueness of the (weak) solution. The sequence of approximating problems is precisely defined in order to avoid this drawback. Moreover, the extension to this general class of cavities, allows us to treat the case of  cavities $D$ touching the boundary $\partial\om$, since in that case $\Omega\backslash D$ may not be Lipschitz even if $\om$ and $D$ are smooth.


Then we deal with the inverse problem of determining cavities $D$  from measurements of the solution $u=u(D)$ of (\ref{probcav1}) on some open arc $\Sigma$ of $\partial\Omega$. Here, we need to assume the cavities $D$ to have Lipschitz boundary. Using uniform $L^{\infty}(\Omega)$ estimates for solutions of  (\ref{probcav1}) and unique continuation properties for elliptic equations, \cite{ARRV}, we prove that one measurement of $u$ on $\Sigma$ is enough to uniquely determine a cavity $D$.

We expect to extend the uniqueness for the inverse problem in the three-dimensional setting in a forthcoming publication.

Let us finally point out that recently there has been  growing interest towards inverse boundary value problems for semilinear elliptic equations, see for example \cite{LLLS1}, \cite{LLLS2},\cite{CMHY},\cite{CF},\cite{KU}. 
In particular, we would like to mention \cite{LLLS1}, where the authors use the nonlinear Dirichlet to Neumann map to recover simultaneously the nonlinear term appearing in the equation and the cavity.
Here, instead, the nonlinearity is given and a collection of separated cavities have to be identified and in fact only one measurement suffices to determine them. Also, we would like to emphasize that our results could be easily extended to the case where
the conormal derivative of $u$ on part of  $\partial\Omega$ is different from zero and  more general nonlinearities are considered.   
\vskip 3truemm
The paper is organized as follows: in Section 2 we state our main assumptions. In Section 3 we analyze the well-posedness of Problem (\ref{probcav}) when $\Omega\backslash D$ is a Lipschitz domain. In Section 4 we extend it to the case when $D$ belongs to a special class of sets of finite perimeter and finally
in Section 5 we prove the uniqueness of the inverse problem in the class of Lipschitz cavities. 


\newpage


\section{Notation and main assumptions}

We consider the following inhomogeneous Neumann problem for a semilinear equation

\begin{equation}
\label{probcav}
\left\{
  \begin{array}{ll}
    -\div(A(x)\nabla u)+u^3=f, & \hbox{in $\omdd$} \\
    \displaystyle{\frac{\partial^A u}{\partial\mathbf{n}}}=0, & \hbox{on $\partial(\omdd)$}.
  \end{array}
\right.
\end{equation}

where we denote by $\displaystyle{\frac{\partial^A u}{\partial\mathbf{n}}}$  the \textit{conormal derivative} of $u$ defined as $A(x)\nabla u\cdot\mathbf{n}$, with $\mathbf{n}$ outer unit normal to $\Omega_D$.

In what follows we'll use the notation 

$$\omd=\omdd$$

Let's now state our main assumptions.
\begin{assumption}\label{as:1}
 $\om\subset \R^2$  is a bounded domain with Lipschitz boundary, $\partial\Omega$. 
\end{assumption}
\begin{assumption}\label{as:2} $\Sigma \subset \partial \Omega$ is an open arc, the portion of boundary which is accessible for measurement.
and  $\Omega_1\subset\Omega$ is such that $\partial\Omega_1$ is Lipschitz and $\Sigma \subset \partial\Omega_1\cap\partial\Omega$.
\end{assumption}
\begin{assumption}\label{as:3}
$D\in\mathcal{D}$ defined by
$$
\mathcal{D}=\{D=\cup_{j=1}^N D_j\, \subset\overline\Omega: D_j \textrm{compact, simply connected}, {\rm int}(D_j)\neq\emptyset, {\rm dist}(D,\overline{\Omega}_1)>0,\, \forall j=1,\dots,N \}
$$
\end{assumption}
Notice that we do not exclude the case that $\partial\Omega\cap\partial D\neq \emptyset$.
\vskip 2truemm
Finally, on the equation parameters we assume
\begin{assumption}\label{as:5}
$A(x)$ is a symmetric matrix of order $2$ satisfying the boundedness and ellipticity conditions 
\begin{equation}
\label{ellipcond}
0<\lambda\|\mathbf{\xi}\|^2\leq\sum_{i,j=1}^2 a_{ij}(x)\xi_i\xi_j\leq\Lambda\|\mathbf{\xi}\|^2, \ \\\ \forall\mathbf{\xi}=(\xi_1,\xi_2)\in\R^2
\end{equation}
where $\lambda,\Lambda$ are positive constants.
\end{assumption}

Assumptions on the source term $f$ will be specified in the sequel when necessary.

\vskip 5truemm




\begin{figure}
\centering
{\includegraphics[width=0.40\textwidth]{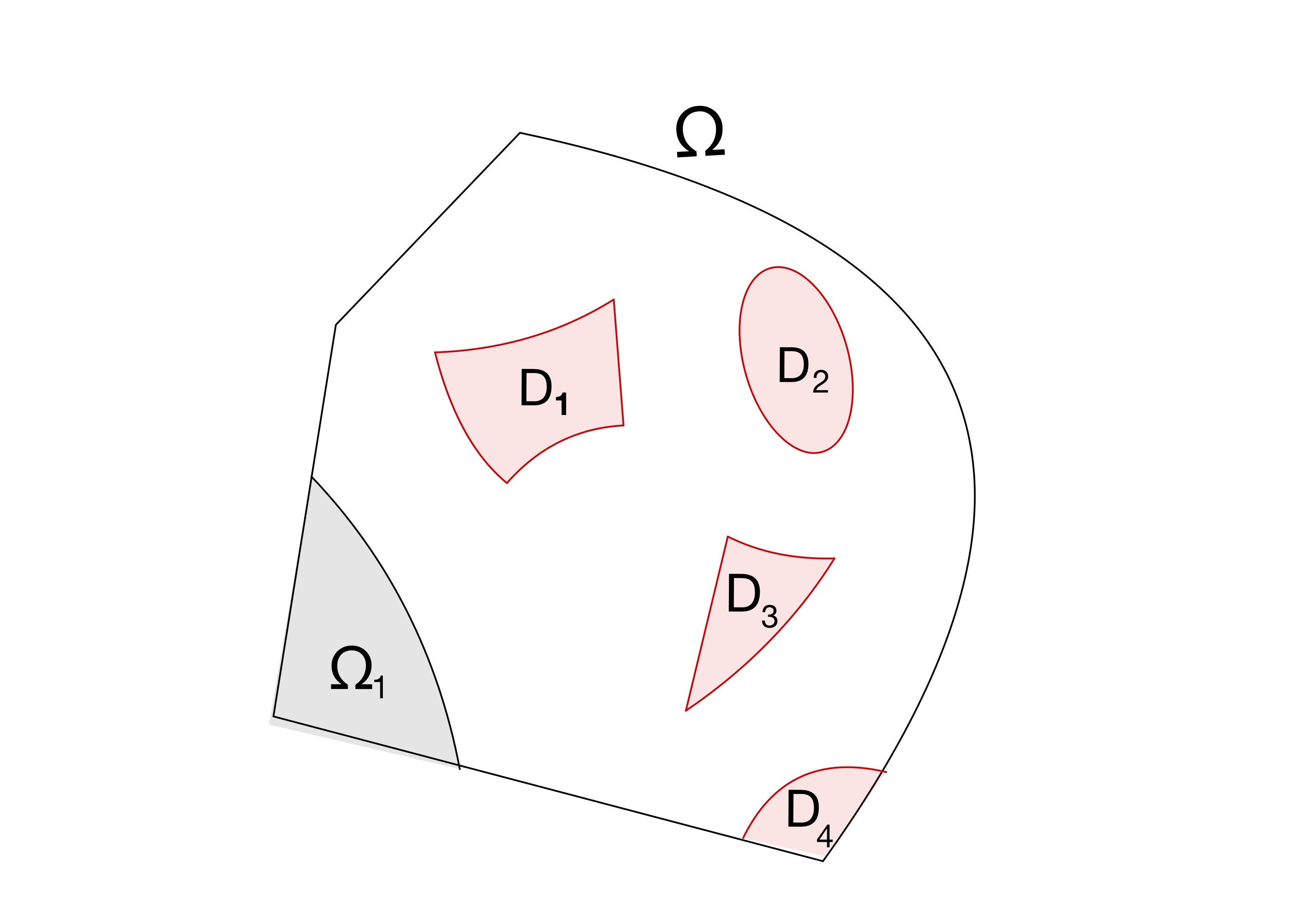}}
\caption{Example of domain with cavities}\label{fig:cav}
\end{figure}

\section{The direct problem - Well posedness when $\omd$ has Lipschitz boun\-da\-ry}

\subsection{Existence and uniqueness of the weak solution}


\smallskip
We will first show that Problem (\ref{probcav}) (in weak form) is well posed under the additional assumption that $\omd$ has Lipschitz boundary. In this case, the result holds for general sources $f\in H^{1}(\omd)'$, the dual space {of $H^1(\omd)$}.

\begin{theorem}
\label{exist}
Assume that $\omd$ has Lipschitz boundary and that $f\in H^{1}(\omd)'$. Then problem \eqref{probcav} has a unique solution $u\in H^1(\omd)$.
\end{theorem}

\begin{proof}
By multiplying the equation in \eqref{probcav} by a test function $\phi$, integrating by parts and using the Neumann boundary condition, we obtain the weak formulation

\begin{equation}
\label{weakfucav}
\int_{\omd} (A(x)\nabla u)\cdot\nabla\phi + \int_{\omd}u^3\phi=\int_{\omd} f\phi  , \quad\quad\forall\phi\in H^1(\omd).
\end{equation}
Now let $T:H^1(\omd)\longrightarrow H^{1}(\omd)'$ be the operator defined by
\begin{equation}
\nonumber
\langle Tu,\phi\rangle =\int_{\omd} (A(x)\nabla u)\cdot\nabla\phi + \int_{\omd}u^3\phi,\quad\quad\forall \phi\in H^1(\omd)
\end{equation}
It is readily verified that $T$ is a \emph{potential operator}, that is
$Tu-f$ is the derivative of the functional
\begin{equation}
\label{funz}
E(u)=\frac{1}{2}\int_{\omd} (A(x)\nabla u)\cdot\nabla u+\frac{1}{4}\int_{\omd}u^4-\int_{\omd} fu
\end{equation}
Then, the theorem will follow by showing that $T$ is \emph{bounded, strictly monotone and coercive}; in fact, by these properties of $T$ the functional $E$ is coercive and weakly lower semicontinuous on $H^1(\omd)$ (see e.g. \cite{fucik}, theorem 26.11). Thus, $E$ is bounded from below and attains its infimum at some $u\in H^1(\omd)$ satisfying
$Tu=f$. The uniqueness of $u$ is a consequence of the strict monotonicity of $T$; for, if $Tu=Tv=f$, equation \eqref{strictmon} below implies $u=v$.

\bigskip

{\bf i.} $T$ is bounded.

\bigskip

By H\"older's inequality and ellipticity condition \eqref{ellipcond}
\begin{equation}
\nonumber
|\langle Tu,\phi\rangle|\leq\Lambda\|\nabla u\|_{L^2(\omd)}\|\nabla\phi\|_{L^2(\omd)} + \|u\|_{L^6(\omd)}^3\|\phi\|_{L^2(\omd)}
\end{equation}

and by Sobolev embedding theorem  $\|u\|_{L^6(\omd)}\leq C_S\|u\|_{H^1(\omd)}$, so that

\begin{equation}
\nonumber
|\langle Tu,\phi\rangle|\leq\Lambda\|\nabla u\|_{L^2(\om)}\|\nabla\phi\|_{L^2(\omd)} + C^3_S\|u\|^3_{H^1(\omd)}\|\phi\|_{L^2(\omd)}\leq\max\left[\Lambda\|u\|_{H^1(\omd)},C^3_S\|u\|_{H^1(\omd)}^3\right]\|\phi\|_{H^1(\omd)}.
\end{equation}

Therefore, if $u$ belongs to a bounded subset of $H^1(\omd)$,

\begin{equation}
\nonumber
\|Tu\|_{H(\omd)'}=\sup_{\phi}\frac{|\langle Tu,\phi\rangle|}{\|\phi\|_{H^1(\omd)}}\leq\max\left[\Lambda\|u\|_{H^1(\omd)},C^3_S\|u\|_{H^1(\omd)}^3\right]=C_2.
\end{equation}

\bigskip

{\bf ii.} $T$ is (strictly) monotone.

\bigskip

\begin{equation}
\nonumber
\langle Tu-Tv,u-v\rangle=\int_{\omd} (A(x)\nabla (u-v))\cdot\nabla (u-v) + \int_{\omd} (u-v)^2(u^2+uv+v^2)\geq0.
\end{equation}

Furthermore

\begin{equation}
\label{strictmon}
\langle Tu-Tv,u-v\rangle=0 \ \ \ \ \ \Leftrightarrow \ \ \ u=v.
\end{equation}

\bigskip

{\bf iii.} $T$ is coercive, that is
\begin{equation}
\label{coerc}
\lim_{\|u\|_{H^1(\om)}\to +\infty}\frac{\langle Tu,u\rangle}{\|u\|_{ H^1(\omd)}}=+\infty
\end{equation}

By using again H\"older's inequality,

\begin{equation}
\nonumber
\langle Tu,u\rangle\geq \lambda\int_{\omd} |\nabla u|^2 + \int_{\omd} u^4\geq \lambda\|\nabla u\|^2_{L^2(\omd)}+\frac{1}{|\omd|}
\left(\int_{\omd}u^2\right)^2\geq \lambda \|\nabla u\|^2_{L^2(\omd)}+\frac{1}{|\omd|}
\|u\|^4_{L^2(\omd)}=
\end{equation}

\begin{equation}
\nonumber
=\lambda\left(\|\nabla u\|^2_{L^2(\omd)}+\|u\|^2_{L^2(\omd)}\right)+\frac{1}{|\omd|}\;
\|u\|^4_{L^2(\omd)}-\lambda\|u\|^2_{L^2(\omd)}.
\end{equation}
\medskip
Finally, by $|\omd|^{-1}x^4-\lambda x^2 \ge -\lambda^2|\omd|/4$,  we get
\begin{equation}
\label{stimacoerc}
\langle Tu,u\rangle\geq \lambda\|u\|^2_{H^1(\omd)}-\frac{\lambda^2}{4}|\omd|
\end{equation}
hence, \eqref{coerc} follows.
\end{proof}

\begin{remark}
It is worthwhile to stress that theorem \ref{exist} also holds when the conormal derivative of $u$ is different from zero on $\partial\Omega_{D}$ .
Actually, in this case the right hand side in \eqref{weakfucav} can be written as $<f,\phi>=\int_{\omd}f_0\phi+\int_{\Gamma}g\phi\,|_{\partial\Omega_{D}}$, with $f_0\in L^2(\omd)$, 
$g\in L^2(\partial\Omega_{D})$.
\end{remark}

\begin{remark}
\label{dim3}
One can easily check that the above proof can be extended when $\omd\subset \mathbb{R}^3$, as the Sobolev embedding  $\|u\|_{L^6(\omd)}\leq C_S\|u\|_{H^1(\omd)}$ still holds.
\end{remark}






\subsection{A priori bounds for the solution}

In this section we will prove estimates of the solutions to Problem (\ref{probcav}) which will be useful in the subsequent discussion. To begin with, we have the following stability estimate:

\begin{proposition}
\label{mainest}
Let $u\in H^1(\omd)$ be a solution of \eqref{probcav}. Then
\begin{equation}
\label{apriori}
\|u\|_{H^1(\omd)}\le \frac{1}{\lambda}\,\|f\|_{(H^1)'}+|\omd |^{1/3}\|f\|^{1/3}_{(H^1)'}
\le C(\|f\|_{(H^1)'}+\|f\|^{1/3}_{(H^1)'})
\end{equation}
where the constant $C=\max\{\frac{1}{\lambda}, |\omd|^{1/3}\}$ and $(H^1)'=H^1(\omd)'$.
\end{proposition}

\begin{proof}
Letting $\phi=u$ in equation \eqref{weakfucav} and by the ellipticity condition  \eqref{ellipcond}, we readily get
\begin{equation}
\label{stim1}
\lambda\|\nabla u\|^2_{L^2(\omd)}+\int_{\omd}u^4\le\|f\|_{(H^1)'}\|u\|_{ H^1(\omd)}.
\end{equation}
By the above inequality we first obtain
\begin{equation}
\nonumber
\|\nabla u\|^2_{L^2(\omd)}\le \frac{\|f\|_{(H^1)'}}{\lambda}\|u\|_{ H^1(\om)}.
\end{equation}
Furthermore, by the inequality
\begin{equation}
\nonumber
\|u\|^4_{L^2(\omd)}\le |\omd|\,\int_{\omd}u^4
\end{equation}
and again by \eqref{stim1} we get
\begin{equation}
\nonumber
\|u\|^2_{L^2(\omd)}\le \big (|\omd|\,\|f\|_{(H^1)'}\|u\|_{ H^1(\omd)}\big )^{1/2}.
\end{equation}
Then
\begin{equation}
\label{estint}
\|u\|^2_{H^1(\omd)}\le \frac{1}{\lambda}{\|f\|_{(H^1)'}}\|u\|_{ H^1(\omd)}+{\big (|\omd|\,\|f\|_{(H^1)'}\big )^{1/2}}\|u\|^{1/2}_{ H^1(\omd)}.
\end{equation}
We can rewrite the above estimate in the form
\begin{equation}
\nonumber
\|u\|^{1/2}_{H^1} \Big(\|u\|_{H^1}-\frac{1}{\lambda}{\|f\|_{(H^1)'}}  \Big)\le |\omd|^{1/2}\|f\|^{1/2}_{(H^1)'}.
\end{equation}
Now, either
$$
\|u\|_{H^1(\omd)}\le \frac{1}{\lambda}\,\|f\|_{(H^1)'}
$$
or
\begin{equation}
\nonumber
\Big (\|u\|_{H^1}-\frac{1}{\lambda}{\|f\|_{(H^1)'}}  \Big )^{3/2}\le |\omd|^{1/2}\|f\|^{1/2}_{(H^1)'}
\end{equation}
In both cases \eqref{apriori} holds.
\end{proof}

\smallskip
Another crucial bound for the solution follows by the maximum principle, with some additional assumption on the source $f$ .

\begin{theorem}
\label{Linfinity}
Let $u\in H^1(\omd)$ be a solution of \eqref{probcav} with $f\in L^2(\omd)$. Then 
\begin{equation}\label{boundsu}
\left(ess\,\inf_{\omd}f\right)^{1/3}\leq u(x) \leq\left(ess\,\sup_{\omd}f\right)^{1/3}\quad\quad \mathrm{a.e.}\quad x\in\omd\,.
\end{equation}
\end{theorem}

\begin{proof}
  We'll use a truncation method due to Stampacchia (see \cite{brezis}). Consider a function $G:\R\rightarrow\R$ such that $G\in C^1(\R)$ and
  \begin{itemize}
    \item $G(t)\geq0$
    \item $G(t)=0$ for $t\leq0$
    \item $0<G^\prime\leq\gamma$ for $t>0$.
  \end{itemize}

  For some $\alpha\in \mathbb{R}$ consider $G(u-\alpha)$; because of the above definition $G(u-\alpha)$ belongs to $H^1(\omd)$ for every $\alpha$ and therefore can be taken as test function in \eqref{weakfucav} to get
\begin{equation}
\label{weakfuG}
\int_{\omd} G^\prime(u-\alpha)(A(x)\nabla u)\cdot\nabla u + \int_{\omd}u^3 G(u-\alpha)=\int_{\omd} f G(u-\alpha)  ;
\end{equation}

subtracting $\displaystyle{\int_{\omd}\alpha^3 G(u-\alpha)}$ from both sides we obtain

 \begin{equation}
\label{weakfuGG}
\int_{\omd} G^\prime(u-\alpha)(A(x)\nabla u)\cdot\nabla u + \int_{\omd}\left(u^3-\alpha^3\right) G(u-\alpha)=\int_{\omd} (f-\alpha^3) G(u-\alpha).
\end{equation}

Now observe that the first integral on the left hand side of \eqref{weakfuGG} is $\geq0$ because of the definition of $G$ and the ellipticity condition; if we now choose $\alpha=\left(\sup_{\omd}f\right)^{1/3}$  the right hand side is $\leq0$. Hence,

$$\int_{\omd}\left(u^3-\alpha^3\right) G(u-\alpha)=
\int_{\omd}(u-\alpha)G(u-\alpha)(u^2+\alpha u+\alpha^2)\leq 0\,.$$

Since $t G(t)\ge 0$ $\,\,\forall t\in \mathbb{R}\,\,$ and also $u^2+\alpha u+\alpha^2\geq0$ the above inequality implies
$(u-\alpha)G(u-\alpha)= 0$ a.e. and therefore $u\leq\alpha=\left(\sup_{\omd}f\right)^{1/3}$ a.e.\,, which is the right hand side inequality in \eqref{boundsu}. For the left hand side inequality apply the same argument to $-u$.

\end{proof}

\begin{remark}
\label{remboundsu}
Observe that in the proof of estimates \eqref{apriori} and  \eqref{boundsu} we did not use any regularity of the set $\Omega_D$.
\end{remark}


\section{The direct problem - Well posedness when $\partial D$ has finite Hausdorff measure}
We now generalize the well-posedness of the direct problem to  the more general class of cavities $D\in\mathcal{D}$ satisfying 
\begin{equation}\label{rectif}
\mathcal{H}(\partial D)<\infty
\end{equation}
where $\mathcal{H}$ denotes the $1$ dimensional Hausdorff measure, $\partial D$ the topological boundary of $D$.

To this aim, we need some additional assumptions on the source term $f$ in Problem (\ref{probcav}):

\smallskip
\begin{assumption}\label{as:6}
\begin{equation}
 f \in L^{\infty}(\Omega),\,\, f\geq 0, \,\,\,\mathrm{supp}(f)\subset\Omega_1,
  \end{equation}
  where $\Omega_1$ is defined in the Assumption 2.
  \end{assumption}
\smallskip
As we will see in the next section, these assumptions are also convenient in investigating uniqueness of the inverse problem. 
\smallskip




\smallskip
We recall that (measurable) subsets whose boundary satisfies \eqref{rectif} are in particular subsets of \emph{finite perimeter} in $\RE^2$ (see e.g. \cite{EG} chapter 5). 
Moreover, by approximation results in geometric measure theory (see e.g. \cite{BP}, \cite{CT}), for every  $D\in \mathcal{D}$
satisfying (\ref{rectif}) there exists a sequence of Lipschitz (in fact, smooth) domains $\Omega_k$  such that $\Omega_1\subset \Omega_k\subset\om_D$ and
\begin{equation}\label{eqdeftilD}
\lim_{k\to\infty}|\Omega_D\setminus \Omega_k|=0\,.
\end{equation}

For any $k\in \mathbf{N}$,  let's now consider the perturbed problem

\begin{equation}
\label{problemvn}
\int_{\omd}\left(A(x)\nabla v\right)\cdot\nabla\phi + \int_{\Omega_D\setminus\Omega_k}v\phi + \int_{\Omega_k}v^3\phi = \int_{\omd}f\phi
\end{equation}

i.e. with a non-linearity of the form

\[
v\mapsto v \chi_{\Omega_D\setminus\Omega_k}+v^3 \chi_{\Omega_k}\,,
\]

where $\chi_E$ denotes the indicator function of a set $E$.

\medskip
The corresponding nonlinear operator (still denoted by $T$) defined by the left hand side of \eqref{problemvn} i. e.

\begin{equation}
\label{opTperfin}
\left\langle Tv,\phi\right\rangle=\int_{\omd}\left(A(x)\nabla v\right)\cdot\nabla\phi + \int_{\Omega_D\setminus\Omega_k}v\phi + \int_{\Omega_k}v^3\phi
\end{equation}

is well defined in $H^1(\omd)$ since the nonlinear term can still be bounded by Sobolev embedding theorem (the boundary of $\om_k$ is Lipschitz). Moreover, $T$ is a potential operator being the derivative of

\begin{equation}
E(v)=\frac{1}{2}\int_{\omd}\left(A(x)\nabla v\right)\cdot\nabla v +\frac{1}{2} \int_{\Omega_D\setminus\Omega_k}v^2+ \frac{1}{4}\int_{\Omega_k}v^4 - \int_{\omd}fv
\end{equation}

Then, following the proof of Theorem \ref{exist} we can still show (see Appendix) that $T$ is \emph{bounded, strictly monotone and coercive}.
It follows that problem \eqref{problemvn} has a unique solution $v_k\in H^1(\omd)$ for every $k$.

Moreover, it turns out that estimates in the previous section also hold for the $v_k$'s; 
note that, in establishing the bound \eqref{boundsuperfin} below, we now exploit the positivity and the localization of the source $f$.
\bigskip

\begin{lemma}
\label{stimesolperfin}
Let $v_k\in H^1(\omd)$ be a solution to \eqref{problemvn}. Then,

\begin{equation}
\label{aprioriperfin}
\|v_k\|_{H^1(\omd)}\le \frac{1}{\lambda}\,\|f\|_{(H^1)'}+|\omd |^{1/3}\|f\|^{1/3}_{(H^1)'}\,.
\end{equation}

Furthermore,  if $f$ satisfies Assumption \ref{as:6}

\begin{equation}\label{boundsuperfin}
0\leq v_k(x) \leq\left(\mathrm{ess}\,\sup_{\omd}f\right)^{1/3}\quad\quad \mathrm{a.e.}\quad x\in\omd\,.
\end{equation}

\end{lemma}

\begin{proof}
Letting $\phi=v_k$ in \eqref{problemvn} and by the ellipticity condition, we get

\begin{equation}
\label{stim1perfin}
\lambda\|\nabla v_k\|^2_{L^2(\omd)}+\int_{\Omega_D\backslash \Omega_k}v_k^2+\int_{\Omega_k}v_k^4
\le\|f\|_{(H^1)'}\|v_k\|_{ H^1(\omd)}.
\end{equation}

In particular

\begin{equation}
\nonumber
\|\nabla v_k\|^2_{L^2(\omd)}+\|v_k\|^2_{L^2(\Omega_D\backslash \Omega_k)}
\le \frac{\|f\|_{(H^1)'}}{\lambda}\|v_k\|_{ H^1(\om)}
\end{equation}

and
\begin{equation}
\nonumber
\|v_k\|^2_{L^2(\om_k)}\le \big (|\omd|\,\|f\|_{(H^1)'}\|v_k\|_{ H^1(\omd)}\big )^{1/2}.
\end{equation}

Summing up the two estimates term by term we get again

\begin{equation}
\nonumber
\|v_k\|^2_{H^1(\omd)}\le \frac{1}{\lambda}{\|f\|_{(H^1)'}}\|v_k\|_{ H^1(\omd)}+{\big (|\omd|\,\|f\|_{(H^1)'}\big )^{1/2}}\|v_k\|^{1/2}_{ H^1(\omd)}
\end{equation}

and the  bound \eqref{aprioriperfin} is proven by the same argument following \eqref{estint}.

\smallskip
Let us now take $f\in L^2(\omd)$ satisfying Assumption \ref{as:6} and consider a function $G$ as in the proof of Theorem \ref{Linfinity}. Then, choose $\phi=G(v_k-\alpha)$ in \eqref{problemvn}, where $\alpha\ge 0$ and get

\begin{equation}
\nonumber
\int_{\omd} G^\prime(v_k-\alpha)(A(x)\nabla v_k)\cdot\nabla v_k + \int_{\Omega_D\backslash \Omega_k}v_k G(v_k-\alpha)
+\int_{\Omega_k}v_k^3 G(v_k-\alpha)=\int_{\Omega_k} f G(v_k-\alpha)\,.
\end{equation}

Subtracting on both sides the term $\displaystyle{\int_{\om_k}\alpha^3G(v_k-\alpha)}$, we obtain the equivalent equation

$$
\int_{\omd} G^\prime(v_k-\alpha)(A(x)\nabla v_k)\cdot\nabla v_k +
\int_{\omd}\big [(v_k-\alpha)\chi_{\Omega_D\backslash \Omega_k}+(v^3_k-\alpha^3)\chi_{\Omega_k}
 \big ]G(v_k-\alpha)=
$$

\begin{equation}
\label{weakfuGperfin}
  = \int_{\Omega_k} (f-\alpha^3) G(v_k-\alpha)-\alpha\int_{\Omega_D\backslash \Omega_k}G(v_k-\alpha)\,.
\end{equation}

By choosing $\alpha=\left(\sup_{\omd}f\right)^{1/3}\ge 0$, the right hand side is $\leq0$.
 Hence,

$$
\int_{\omd}\big [(v_k-\alpha)\chi_{\Omega_D\backslash \Omega_k}+(v^3_k-\alpha^3)\chi_{\Omega_k}
 \big ]G(v_k-\alpha)=
 $$

$$
\int_{\omd}(v_k-\alpha)\big [\chi_{\Omega_D\backslash \Omega_k}+(v^2_k+v_k\alpha+\alpha^2)\chi_{\Omega_k}
 \big ]G(v_k-\alpha)\le 0\,.
$$

Since the term in the square brackets is $\ge \min\{1, 3\alpha^2/4\}$ in $\omd$, we conclude as in Theorem \ref{Linfinity} that
$v_k\leq\alpha=\left(\sup_{\omd}f\right)^{1/3}$ a.e.\, in $\omd$.

Finally, the bound from below follows by applying the same arguments to $-v_k$.
\end{proof}

By the previous  estimates we can now prove:
\begin{theorem}
\label{existfinper}
Assume $D\in \mathcal{D}$ satisfying (\ref{rectif}) and let $f$ satisfy Assumption \ref{as:6}.
Then problem \eqref{probcav} has a unique solution $v\in H^1(\omd)\cap L^{\infty}(\omd)$. Furthermore, $v$ satisfies the bounds \eqref{aprioriperfin}, \eqref{boundsuperfin}.
\end{theorem}

\begin{proof}
Let $\{\om_k\}\subset \Omega$ be a sequence of domains chosen as at the beginning of the section; note that
$\,\mathrm{supp}\,f\subset\Omega_k$.

Thus, Lemma \ref{stimesolperfin} holds and the solutions ${v_k}$, $k\in \mathbf{N}$, to \eqref{problemvn}
are uniformly bounded in $H^1(\omd)$.  Hence,  there is a subsequence (still denoted by $v_k$) and a function $v$ such that
$\,v_k\rightharpoonup v\,$ in  $\,H^1(\omd)$ and a.e. in $\omd$. Moreover, $\|v_k\|_{L^{\infty}(\omd)}$ is uniformly bounded by $\|f\|_{L^{\infty}(\omd)}^{1/3}$.

Then, by writing \eqref{problemvn} in the form

\begin{equation}
\nonumber
\int_{\omd}\left(A(x)\nabla v_k\right)\cdot\nabla\phi + \int_{\omd}v_k^3\phi
+\int_{\Omega_D\setminus\Omega_k}(v_k-v_k^3)\phi   = \int_{\omd}f\phi\,,
\end{equation}


we easily get, by the weak convergence and the dominated convergence theorem,

\begin{equation}
\label{problemvn2}
\int_{\omd}\left(A(x)\nabla v\right)\cdot\nabla\phi + \int_{\omd}v^3\phi
 = \int_{\omd}f\phi\,,
\end{equation}

that is, $v$ solves \eqref{weakfucav}. Uniqueness follows from the fact that the operator defined by the left hand side of \eqref{problemvn2} is strictly monotone in $H^1(\omd)$.

Finally, since $\|v\|_{H^1(\omd)}\le \liminf \|v_k\|_{H^1(\omd)}$ and $v_k\rightarrow v$ a.e. in $\omd$, the bounds
\eqref{aprioriperfin} and \eqref{boundsuperfin} also hold for the solution $v$.

\end{proof}

\begin{remark}
\label{boundsolperfin1}
Theorem \ref{existfinper} and the above remark apply in particular to the problem with a cavity $D$ such that
$\partial D\cap\partial\om\neq\emptyset$. Note that in this case the boundary
of $\omd$ might not be Lipschitz even if $\partial\om$ and $\partial D$ are regular.
\end{remark}

\begin{remark}
\label{boundsolperfin2}
Obviously, the results in this section hold for any $D\in\mathcal{D}$ such that $\om_D$ is the limit (in measure) of Lipschitz domains $\Omega_k$, with $\Omega_1\subset \Omega_k\subset\om_D$. It is not difficult to check that such $D$'s have finite perimeter. 

The extension to the whole class of finite perimeter subsets $D\in\mathcal{D}$ remains an open problem.
\end{remark}

\section{Uniqueness of the inverse problem}


In this section we investigate the uniqueness of the inverse problem, i. e.
\begin{problem}
\label{inverseproblem}
\label{ext} Assume it is possible to measure the solution to (\ref{probcav}), $u\big|_\Sigma$ where $\Sigma$ is an open connected portion of $\partial\Omega_1\cap\partial\Omega$. Is it possible to uniquely determine $D$? 

\end{problem}
The answer is yes assuming the coefficient matrix $A$ has constant entries.

\begin{theorem}
Let Assumptions  \ref{as:1}-\ref{as:6} hold, $A$ constant matrix and $D_1$, $D_2$ in $\mathcal{D}$ where $D_1=\cup_{j=1}^N D_1^j $  and $D_2=\cup_{j=1}^M D_2^j$ and assume  $D_1,D_2$ with Lipschitz boundary.

Let  $u_1$ and $u_2$ be the corresponding solutions to the boundary value problem \eqref{probcav}.

Then $u_1\big|_\Sigma=u_2\big|_\Sigma$ implies $D_1\equiv D_2$ (which in particular implies $N=M$).

\end{theorem}
\begin{proof}







Without loss of generality we will prove the result for the Laplacian.
Assume first that the cavities are strictly contained in $\Omega$. 
Let $u_1=u_2$ on $\Sigma$;
we will prove the theorem through an argument by contradiction.  Assume $D_1\neq D_2$ then setting $w:=u_1-u_2$ we have that

\begin{equation}\label{}
  w\big|_\Sigma=0 \ \ \ \ \ \ {\rm and} \ \ \ \ \ \ \displaystyle{\nabla w \cdot \nu \big|_\Sigma}=0
\end{equation}

Moreover $w$ is a solution to

\begin{equation}\label{geneqw}
-\Delta w+ q(x)\,w=0 \ \ \ \ \hbox{in $\Omega \setminus( D_1\cup D_2)$}
\end{equation}

where $q(x)=u_1^2+u_1 u_2+u_2^2$. By (\ref{boundsu}) we have
$$\|q\|_{L^\infty(\Omega\setminus( D_1\cup D_2))}\leq C,$$
with $C$ a constant depending on $\|f\|_{L^\infty(\Omega_D)}$. Uniqueness for the Cauchy problem  together with  the weak unique continuation property (see for example \cite{ARRV})  implies

\begin{equation}\label{}
  w\equiv0 \ \ \ \ \ \ {\rm in} \ \ \ \ \ G
\end{equation}

where $G$ is the connected component of $\overline{\Omega}\setminus D_1\cup D_2$ that contains $\Sigma$. Let $\tilde{G}=\Omega\setminus G$ and observe that:

$\tilde{G}$ is closed, $\tilde{G}\supseteq D_1\cup D_2$ and

\begin{equation}\label{tildGbd}
\partial\tilde{G}=\big (\partial D_1\cup \partial D_2\big )\cap \partial G\,.
\end{equation}

 Let $\tilde{D}$ be a connected component of $\tilde{G}\setminus D_2$
 (note that $\tilde{D}=D_1$ if $D_1\cap D_2=\emptyset$). Then we have

\begin{equation}\label{tildDbd}
\partial\tilde D\subseteq \partial\big (\tilde{G}\setminus D_2 \big )\subseteq\partial\tilde G\cup\partial D_2\,.
\end{equation}

We further note that, unless $D_1\subset D_2$, we may assume that $\tilde{D}$ contains a subset of $D_1$ with nonempty interior. Otherwise, we just exchange the roles of $D_1$ and $D_2$.

\smallskip
Let us now define
$\Gamma_1\equiv\partial\tilde{D}\cap\partial D_1\cap\partial\tilde{G}$ and let $\partial\tilde{D}=\Gamma_1\cup \Gamma_2$. Observe that (\ref{tildGbd}) implies that $\Gamma_1\subset\partial D_1\cap \partial G$ and from (\ref{tildDbd}) $\Gamma_2\equiv\partial\tilde{D}\setminus\Gamma_1\subset\partial D_2$ including possibly the case where $\Gamma_2=\emptyset$.
\smallskip
In figures \ref{fig:ins1} and \ref{fig:ins2} the shaded regions show examples of possible geometries.

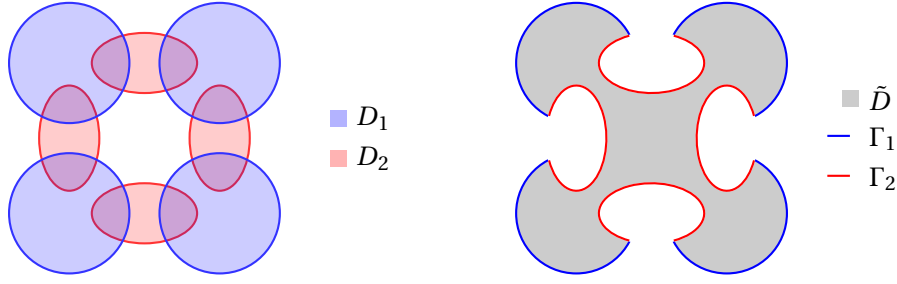
\begin{figure}
  \centering
  {		    	
\begin{tikzpicture}
    \draw [red!80,fill=red!20,thick] (-1,0) ellipse (0.4 and 0.7);
    \draw [red!80,fill=red!20,thick] (1,0) ellipse (0.4 and 0.7);
    \draw [red!80,fill=red!20,thick] (0,-1) ellipse (0.7 and 0.4);
    \draw [red!80,fill=red!20,thick] (0,1) ellipse (0.7 and 0.4);
    
    \draw [blue!80,fill=blue,fill opacity=.2,thick] (-1,1) circle (0.8);
    \draw [blue!80,fill=blue,fill opacity=.2,thick] (1,1) circle (0.8);
    \draw [blue!80,fill=blue,fill opacity=.2,thick] (1,-1) circle (0.8);
    \draw [blue!80,fill=blue,fill opacity=.2,thick] (-1,-1) circle (0.8);
    
    \matrix [left] at (3.5,0) {
  \node [shape=rectangle,fill=blue!30,label=right:$D_1$] {}; \\
  \node [shape=rectangle,fill=red!30,label=right:$D_2$] {}; \\
};
\end{tikzpicture}
\hspace{0.075\textwidth}
\begin{tikzpicture}
    \draw [blue,fill=black!20,thick](-1,1) circle (0.8);
    \draw [blue,fill=black!20,thick] (1,1) circle (0.8);
    \draw [blue,fill=black!20,thick] (1,-1) circle (0.8);
    \draw [blue,fill=black!20,thick] (-1,-1) circle (0.8);
    \draw [black!20,fill=black!20] (0,0) circle (0.8);
    
    \draw [white,fill=white,thick] (-1,0) ellipse (0.4 and 0.7);
    \draw [white,fill=white,thick] (1,0) ellipse (0.4 and 0.7);
    \draw [white,fill=white,thick] (0,-1) ellipse (0.7 and 0.4);
    \draw [white,fill=white,thick] (0,1) ellipse (0.7 and 0.4);
    
    \draw[thick, red] (-1,0) [partial ellipse=-155:155:0.4 and 0.7];
    \draw[thick, red] (1,0) [partial ellipse=25:335:0.4 and 0.7];
    \draw[thick, red] (0,1) [partial ellipse=115:425:0.7 and 0.4];
    \draw[thick, red] (0,-1) [partial ellipse=-65:245:0.7 and 0.4];
    
    \matrix [left] at (3.5,0) {
  \node [shape=rectangle,fill=black!20,label=right:$\tilde{D}$] {}; \\
  \node [label=right:$\Gamma_1$] {\blueline}; \\
  \node [label=right:$\Gamma_2$] {\redline}; \\
};
\end{tikzpicture}
}
  \caption{Possible configuration: example 1}\label{fig:ins1}
\end{figure}

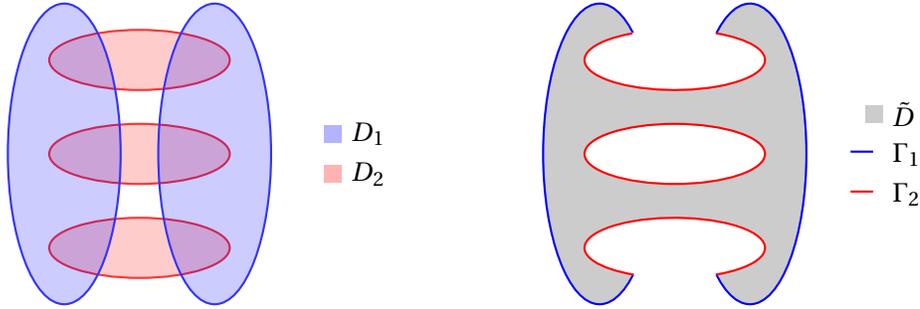
\begin{figure}
  \centering
  {		    	
\begin{tikzpicture}
    \draw [red!80,fill=red!20,thick] (0,1.25) ellipse (1.2 and 0.4);
    \draw [red!80,fill=red!20,thick] (0,0) ellipse (1.2 and 0.4);
    \draw [red!80,fill=red!20,thick] (0,-1.25) ellipse (1.2 and 0.4);
    
    \draw [blue!80,fill=blue,fill opacity=.2,thick] (-1,0) ellipse 
    (0.75 and 2);
    \draw [blue!80,fill=blue,fill opacity=.2,thick] (1,0) ellipse (0.75 and 2);
    
    \matrix [left] at (3.5,0) {
  \node [shape=rectangle,fill=blue!30,label=right:$D_1$] {}; \\
  \node [shape=rectangle,fill=red!30,label=right:$D_2$] {}; \\
};
\end{tikzpicture}
\hspace{0.1\textwidth}
\begin{tikzpicture}
    \draw [blue,fill=black!20,thick] (-1,0) ellipse (0.75 and 2);
    \draw [blue,fill=black!20,thick] (1,0) ellipse (0.75 and 2);
    \draw [black!20,fill=black!20] (0,0) circle (1);
    
    \draw [white,fill=white] (0,1.25) ellipse (1.2 and 0.4);
    \draw [white,fill=white] (0,0) ellipse (1.2 and 0.4);
    \draw [white,fill=white] (0,-1.25) ellipse (1.2 and 0.4);
   
    \draw[thick, red] (0,1.25) [partial ellipse=118:423:1.2 and 0.4];
    \draw[thick, red] (0,0) [partial ellipse=0:360:1.2 and 0.4];
    \draw[thick, red] (0,-1.25) [partial ellipse=-63:243:1.2 and 0.4];

    \matrix [left] at (3.5,0) {
  \node [shape=rectangle,fill=black!20,label=right:$\tilde{D}$] {}; \\
  \node [label=right:$\Gamma_1$] {\blueline}; \\
  \node [label=right:$\Gamma_2$] {\redline}; \\
};

\end{tikzpicture}
}
  \caption{Possible configuration: example 2}\label{fig:ins2}
\end{figure}

\begin{itemize}
\item Assume first that $\tilde{D}$ is Lipschitz.
Observe that on $\partial\tilde{D}$ we have $\displaystyle{\frac{\partial u_2}{\partial\mathbf{n}}=0}$ a.e..  Actually,  on $\Gamma_1$ we have
$\displaystyle{\frac{\partial u_1}{\partial\mathbf{n}}=0}$ a.e.  and observing that on $\Gamma_1\subset\partial G$   $\displaystyle{\frac{\partial u_1}{\partial\mathbf{n}}}=\frac{\partial u_2}{\partial\mathbf{n}}$ a. e.,  it follows that  $\displaystyle{\frac{\partial u_2}{\partial\mathbf{n}}=0}$ a.e. on $\Gamma_1$.  Moreover, on $\Gamma_2$ we have that
$\displaystyle{\frac{\partial u_2}{\partial\mathbf{n}}=0}$ a.e.\\
Also,  $u_2\in H^1(\tilde{D})$ is solution of
 \begin{equation}\label{geneq1}
  -\Delta u_2 + u^3_2=0 \ \ \ \ \hbox{in $\tilde{D}$}
\end{equation}
since $\tilde{D}\cap\Omega_1=\emptyset$. Hence,  multiplying the equation by $u_2$ and integrating by parts over the set $\tilde{D}$, we have
 \begin{equation}\label{geneq2}
  \int_{\tilde{D}}\nabla u_2\cdot \nabla u_2 +u^4_2 =0,
\end{equation}
which implies that necessarily $u_2=0$ a.e. in $\tilde{D}$ and, by unique continuation,  also in
$G\setminus {\rm supp}\,f$.
Now, let $K\subset\ G$ and $K\supset{{\rm supp}\,f}$ be a smooth domain. Then letting $v_2=u_2^3$,  $u_2$ satisfies

\begin{equation}
\label{pbK}
\left\{
  \begin{array}{ll}
    -\Delta u_2+v_2 u_2=f, & \hbox{in $K$} \\
    \displaystyle{\frac{\partial u_2}{\partial\mathbf{n}}}=0, & \hbox{on $\partial K$} \\
    u_2=0, & \hbox{on $\partial K$} ,
  \end{array}
\right.
\end{equation}

Consider $\psi$, any solution of the Schr\"odinger equation

\begin{equation}
    -\Delta\psi+v_2\psi=0, \ \ \ \hbox{in $K$};
\end{equation}

we'll show that $\int_{K} f\,\psi\,=0$.

To see this multiply equation in \eqref{pbK} by $\psi$, integrate twice by parts and use conditions on $u_2$ in \eqref{pbK}

\begin{align}\nonumber
\int_{K} f\,\psi\,=\int_{K}\big[-\Delta u_2\psi+v_2 u_2 \psi\big]=-\int_{\partial K}\frac{\partial u_2}{\partial\mathbf{n}}\,\psi +\int_{K}\big[\nabla u_2\cdot\nabla\psi+v_2 u_2\psi\big]=\\
=\int_{\partial K}u_2\,\frac{\partial \psi}{\partial\mathbf{n}} +\int_{K}\big[-\Delta\psi u_2+v_2 u_2 \psi\big]=\int_{K}u_2\big[-\Delta\psi +v_2 \psi\big]=0.
\end{align}

Now choose $\psi$ such that $\psi\big|_{\partial K}=\alpha<0$: then maximum principle implies that $\psi<0$ in $K$ and from $\displaystyle{\int_{K} f\,\psi}=0$
with $f\geq 0$; we conclude that $f\equiv 0$ in $K$ which contradicts the initial hypotheses.

 \item Now consider the case where $\tilde{D}$ is not  Lipschitz.  Observe that singularities on $\partial\tilde{D}$ can only occur at points of intersection between $\Gamma_1$ and $\Gamma_2$. Let's name them $P^h$, $h=1,\dots,L$. Note that singularities can only be cusps since $\Gamma_1$ and $\Gamma_2$ are Lipschitz.
We will now construct an approximation of $\tilde{D}$ with a sequence of Lipschitz domains, $\{\tilde{D}_k\}\subset\tilde{D}$. For each $P^j$ take two sequences $\{P_k^h\}\subset\Gamma_1$ and $\{{Q}_k^h\}\subset\Gamma_2$ such that both the arcs $\widearc{{P^h}{P_k^h}} \subset\Gamma_1$ and $\widearc{P^h {Q}_k^h}\subset\Gamma_2$ have length which tends to zero as $k\rightarrow +\infty$. For each $h=1,\dots, L$ let $\gamma_k^h$ be smooth arcs joining  $P_k^h$ and ${Q}_k^h$, so that the region $\tilde{D}_k$ enclosed by the $(\gamma_k^h)$'s and $\Gamma_1$ and $\Gamma_2$ is Lipschitz. Moreover assume that $\textrm{length}(\gamma_k^h)\rightarrow 0$  and $|\tilde{D}\backslash\tilde{D}_k|\rightarrow 0$ as $k\rightarrow +\infty$.

Again  by Assumption \eqref{as:6}, $u_2\in H^1(\tilde{D}_k)$ is a weak solution of
 \begin{equation}\label{geneq3}
-\Delta u_2+ u_2^3=0 \ \ \ \ \hbox{in $\tilde{D}_k$}
\end{equation}
with $\displaystyle{\frac{\partial u_2}{\partial\mathbf{n}}=0}$ a.e. on $ \displaystyle \partial\tilde{D}_k\backslash \left (\bigcup_{h=1}^L\gamma_k^h\right)$. Hence, multiplying equation (\ref{geneq3}) by $u_2$,  integrating by parts over $\tilde{D}_k$ and using the fact that  $u_2 \in H^{3/2}(\Omega_{D_2})$, see for example \cite{JK}, \cite{C}, it follows that
\begin{equation}\label{geneqk}
 \int_{\tilde{D}_k}\nabla u_2\cdot \nabla u_2 +u^4_2(x)=\sum_{h=1}^L\int_{\gamma_k^h}\frac{\partial u_2}{\partial\mathbf{n}}u_2.
\end{equation}
Moreover, since $\tilde{D}_k$ is a Lipschitz domain and $u_2$ is uniformly bounded in $\Omega\backslash D_2$  it follows that
$$
\left|\int_{\gamma_k^h}\frac{\partial u_2}{\partial\mathbf{n}}u_2\right|\leq \left(\int_{\gamma_k^h}\left|\frac{\partial u_2}{\partial\mathbf{n}}\right|^2\right)^{1/2}\left(\int_{\gamma_k^h}|u_2|^2\right)^{1/2}\leq C ||u_2||_{H^{3/2}(\Omega_{D_2})}(\textrm{length}(\gamma^h_k))^{1/2}
$$

for $h=1,...,L$. So, passing to the limit as $k\rightarrow +\infty$ in (\ref{geneqk}) we finally get
$$
 \int_{\tilde{D}}\nabla u_2\cdot \nabla u_2 +u^4_2(x)=0
$$
which again gives $u_2=0$ a.e. in $\tilde{D}$ and, by unique continuation, also in $\left(\om\setminus D_1\cup D_2\right)\setminus{{\rm supp}\,f}$.

From now on we we can proceed as in the previous case and conclude the proof.

 \end{itemize}

Finally, if the cavity touches the boundary of $\Omega$ forming a cusp, it is sufficient to modify the above estimates by replacing the norm 
$\|u \|_{H^{3/2}(\Omega_{D_2})}$ with $\|u \|_{H^{3/2}(\Omega^L_{D_2})}$, where
$\Omega^L_{D_2}$ is any Lipschitz subset satisfying 
$\tilde{D}\subset\Omega^L_{D_2}\subset \Omega_{D_2}$. In fact, it can be shown that the solution $u_2$ (which is uniquely defined in $\Omega_{D_2}$ by theorem \ref{existfinper})
has the required regularity in those subsets.



\end{proof}

\begin{remark} Note that we use the assumption of constant coefficient matrix $A$ only in the second part of the proof and for the case of cavities touching the boundary since we need to guarantee $H^{3/2}$ regularity of solutions. It would be interesting to see whether it is possible to extend  uniqueness to variable matrices $A(x)$.
\end{remark}

\section{Conclusions}
In this paper we analyze the inverse problem of determining insulating planar regions, {\it cavities}, from boundary measurements in a nonlinear elliptic equation which arises in cardiac electrophysiology. 
We first show well-posedness of the forward problem in the class of cavities with boundaries having finite one-dimensional Hausdorff measure. Then we prove unique determination of multiple Lipschitz cavities from knowledge of a single boundary measurement of the potential. This last result has been obtained using unique continuation properties of solutions to linear elliptic equations and a-priori bounds on the solution to the forward problem. \\
We plan to analyze also the three-dimensional case.\\
This paper is also the starting point for the implementation of a reconstruction algorithm based on a phase field approach that will be presented in a forthcoming publication. 

Although the proof of uniqueness of the solution to the inverse problem is independent on the nonlinear structure of the equation, it would be interesting to make a comparison between the reconstruction of the cavity in the nonlinear problem  with the one obtained in the linear conductivity equation. 

\section{Appendix}

We show that the estimates in the proof of theorem \ref{exist} can be repeated to show that also the operator $T$ defined by \eqref{opTperfin} is \emph{bounded, strictly monotone and coercive}.

\bigskip
{\bf i.} $T$ is bounded.

\bigskip

\begin{equation}
\nonumber
|\langle Tu,\phi\rangle|\leq\Lambda\|\nabla u\|_{L^2(\omd)}\|\nabla\phi\|_{L^2(\omd)} +\|u\|_{L^2(\omd\setminus \om_k)}\|\phi\|_{L^2(\omd\setminus \om_k)}+ \|u\|_{L^6(\om_k)}^3\|\phi\|_{L^2(\om_k)}\leq
\end{equation}

\begin{equation}
\nonumber
\leq\Lambda\|\nabla u\|_{L^2(\omd)}\|\nabla\phi\|_{L^2(\omd)} +\|u\|_{L^2(\omd\setminus \om_k)}\|\phi\|_{L^2(\omd\setminus \om_k)}+ C^3_S\|u\|^3_{H^1(\om_k)}\|\phi\|_{L^2(\om_k)}\leq\end{equation}

\begin{equation}
\le\max\left[(\Lambda+1)\|u\|_{H^1(\omd)},C^3_S\|u\|_{H^1(\omd)}^3\right]\|\phi\|_{H^1(\omd)}.
\end{equation}

Therefore, if $u$ belongs to a bounded subset of $H^1(\omd)$,

\begin{equation}
\nonumber
\|Tu\|_{(H^1)'(\omd)}=\sup_{\phi}\frac{|\langle Tu,\phi\rangle|}{\|\phi\|_{H^1(\omd)}}\leq\max\left[\Lambda\|u\|_{H^1(\omd)},(1+C^3_S)\|u\|_{H^1(\omd)}^3\right]=C_4.
\end{equation}

\bigskip

{\bf ii.} $T$ is (strictly) monotone.

\bigskip

\begin{equation}
\nonumber
\langle Tu-Tv,u-v\rangle=\int_{\omd} (A(x)\nabla (u-v))\cdot\nabla (u-v) +\int_{\omd\setminus \om_k} (u-v)^2+ \int_{\om_k} (u-v)^2(u^2+uv+v^2)\geq 0.
\end{equation}

and

\begin{equation}
\label{strictmonnonl}
\langle Tu-Tv,u-v\rangle=0 \ \ \ \ \ \Leftrightarrow \ \ \ u=v.
\end{equation}

\bigskip

{\bf iii.} $T$ is coercive

\begin{equation}
\nonumber
\langle Tu,u\rangle\geq \lambda\int_{\omd} |\nabla u|^2 +\int_{\omd\setminus \om_k}u^2+ \int_{\om_k} u^4\geq
\end{equation}

\begin{equation}
\geq \lambda\|\nabla u\|^2_{L^2(\omd)}+\|u\|^2_{L^2(\omd\setminus \om_k)}+\frac{1}{|\om_k|}
\|u\|^4_{L^2(\om_k)}\geq
\end{equation}

\begin{equation}
\nonumber
\geq\min(1,\lambda)\|u\|^2_{H^1(\omd)}+\frac{1}{|\om_k|}\;
\|u\|^4_{L^2(\om_k)}-\|u\|^2_{L^2(\Omega_k)}.
\end{equation}

\medskip
Finally, as in the previous proof from $|\om_k|^{-1}x^4- x^2 \ge -|\om_k|/4$,  with $x=\|u\|_{L^2(\om_k)}$, we get

\begin{equation}
\label{stimacoercnLip}
\langle Tu,u\rangle\geq \tilde{\lambda}\|u\|^2_{H^1(\omd)}-\frac{1}{4}|\om_k|\,,
\end{equation}
where $\tilde\lambda:=\min(1,\lambda)$;
 hence, \eqref{coerc} follows.

\end{document}